\documentclass[reqno, 12pt]{amsart}
\pdfoutput=1
\makeatletter
\let\origsection=\section \def\section{\@ifstar{\origsection*}{\mysection}}
\def\mysection{\@startsection{section}{1}\z@{.7\linespacing\@plus\linespacing}{.5\linespacing}{\normalfont\scshape\centering\S}}
\makeatother

\usepackage{amsmath,amssymb,amsthm}
\usepackage{mathrsfs}
\usepackage{mathabx}\changenotsign
\usepackage{dsfont}
\usepackage{bbm}

\usepackage{xcolor}
\usepackage[backref]{hyperref}
\hypersetup{
    colorlinks,
    linkcolor={red!60!black},
    citecolor={green!60!black},
    urlcolor={blue!60!black}
}

\usepackage[open,openlevel=2,atend]{bookmark}

\usepackage[abbrev,msc-links,backrefs]{amsrefs}
\usepackage{doi}

\renewcommand{\PrintDOI}[1]{\doi{#1}}

\usepackage[T1]{fontenc}
\usepackage{lmodern}
\usepackage[babel]{microtype}
\usepackage[english]{babel}
\usepackage{verbatim}

\usepackage{geometry}
\numberwithin{equation}{section}
\numberwithin{figure}{section}

\usepackage{enumitem}

\let\polishlcross=\l
\def\l{\ifmmode\ell\else\polishlcross\fi}

\let\emptyset=\varnothing
\let\setminus=\smallsetminus

\makeatletter
\def\moverlay{\mathpalette\mov@rlay}
\def\mov@rlay#1#2{\leavevmode\vtop{   \baselineskip\z@skip \lineskiplimit-\maxdimen
   \ialign{\hfil$\m@th#1##$\hfil\cr#2\crcr}}}
\newcommand{\charfusion}[3][\mathord]{
    #1{\ifx#1\mathop\vphantom{#2}\fi
        \mathpalette\mov@rlay{#2\cr#3}
      }
    \ifx#1\mathop\expandafter\displaylimits\fi}
\makeatother

\newcommand{\dcup}{\charfusion[\mathbin]{\cup}{\cdot}}

\DeclareFontFamily{U}  {MnSymbolC}{}
\DeclareSymbolFont{MnSyC}         {U}  {MnSymbolC}{m}{n}
\DeclareFontShape{U}{MnSymbolC}{m}{n}{
    <-6>  MnSymbolC5
   <6-7>  MnSymbolC6
   <7-8>  MnSymbolC7
   <8-9>  MnSymbolC8
   <9-10> MnSymbolC9
  <10-12> MnSymbolC10
  <12->   MnSymbolC12}{}
\DeclareMathSymbol{\powerset}{\mathord}{MnSyC}{180}

\usepackage{tikz}
\usetikzlibrary{calc,decorations.pathmorphing}
\pgfdeclarelayer{background}
\pgfdeclarelayer{foreground}
\pgfdeclarelayer{front}
\pgfsetlayers{background,main,foreground,front}

\usepackage{subcaption}
\newcommand{\qedge}[7]{

	\ifx\relax#4\relax
		\def\qoffs{0pt}
	\else
		\def\qoffs{#4}
	\fi

	\def\qhedge{
		($#1+#3!\qoffs!-90:#2-#3$) --
		($#2+#1!\qoffs!-90:#3-#1$) --
		($#3+#2!\qoffs!-90:#1-#2$) -- cycle}

	\coordinate (12) at ($#1!\qoffs!90:#2$);
	\coordinate (13) at ($#1!\qoffs!-90:#3$);
	\coordinate (23) at ($#2!\qoffs!90:#3$);
	\coordinate (21) at ($#2!\qoffs!-90:#1$);
	\coordinate (31) at ($#3!\qoffs!90:#1$);
	\coordinate (32) at ($#3!\qoffs!-90:#2$);
	
	\def\nqhedge{
		(13) let \p1=($(13)-#1$), \p2=($(12)-#1$) in
			arc[start angle={atan2(\y1,\x1)}, delta angle={atan2(\y2,\x2)-atan2(\y1,\x1)-360*(atan2(\y2,\x2)-atan2(\y1,\x1)>0)}, x radius=\qoffs, y radius=\qoffs] --
		(21) let \p1=($(21)-#2$), \p2=($(23)-#2$) in
			arc[start angle={atan2(\y1,\x1)}, delta angle={atan2(\y2,\x2)-atan2(\y1,\x1)-360*(atan2(\y2,\x2)-atan2(\y1,\x1)>0)}, x radius=\qoffs, y radius=\qoffs] --
		(32) let \p1=($(32)-#3$), \p2=($(31)-#3$) in
			arc[start angle={atan2(\y1,\x1)}, delta angle={atan2(\y2,\x2)-atan2(\y1,\x1)-360*(atan2(\y2,\x2)-atan2(\y1,\x1)>0)}, x radius=\qoffs, y radius=\qoffs] --
		cycle}

		\ifx\relax#5\relax
		\def\qlwidth{1pt}
	\else
		\def\qlwidth{#5}
	\fi
	
		\ifx\relax#7\relax
		\fill \nqhedge;
	\else
		\fill[#7]\nqhedge;
	\fi

		\ifx\relax#6\relax
		\draw[dotted, line width=\qlwidth,rounded corners=\qoffs]\nqhedge;
	\else
		\draw[dotted, line width=\qlwidth,#6]\nqhedge;
	\fi
}

\let\epsilon=\varepsilon

\let\rho=\varrho
\let\theta=\vartheta

\newtheoremstyle{note}  {4pt}  {4pt}  {\sl}  {}  {\bfseries}  {.}  {.5em}          {}
\newtheoremstyle{introthms}  {3pt}  {3pt}  {\itshape}  {}  {\bfseries}  {.}  {.5em}          {\thmnote{#3}}
\newtheoremstyle{remark}  {2pt}  {2pt}  {\rm}  {}  {\bfseries}  {.}  {.3em}          {}

\theoremstyle{plain}
\newtheorem{theorem}{Theorem}[section]
\newtheorem{lemma}[theorem]{Lemma}

\newtheorem{constr}[theorem]{Construction}
\newtheorem{conj}[theorem]{Conjecture}

\newtheorem{claim}[theorem]{Claim}

\newtheorem{problem}[theorem]{Problem}

\theoremstyle{note}

\theoremstyle{remark}
\newtheorem{remark}[theorem]{Remark}

\usepackage{accents}

\usepackage{lineno}
\newcommand*\patchAmsMathEnvironmentForLineno[1]{
\expandafter\let\csname old#1\expandafter\endcsname\csname #1\endcsname
\expandafter\let\csname old#1\expandafter\endcsname\csname end#1\endcsname
\renewenvironment{#1}
{\linenomath\csname old#1\endcsname}
{\csname oldend#1\endcsname\endlinenomath}}

\AtBeginDocument{
}

\begin{document}

\title[On extremal problems concerning the traces of sets]{On extremal problems concerning the traces of sets}
\author[S. Piga]{Sim\'on Piga}
\address{Fachbereich Mathematik, Universit\"{a}t Hamburg, Hamburg, Germany}
\email{simon.piga@uni-hamburg.de}
\author[B.~Sch\"{u}lke]{Bjarne Sch\"{u}lke}
\address{Fachbereich Mathematik, Universit\"{a}t Hamburg, Hamburg, Germany}
\email{bjarne.schuelke@uni-hamburg.de}
\thanks{The first author was supported by ANID/CONICYT Acuerdo Bilateral DAAD/62170017 through a Ph.D. Scholarship. The second author was partially supported by G.I.F. Grant Agreements No. I-1358-304.6/2016.}


\begin{abstract}
Given two non-negative integers~$n$ and~$s$, define~$m(n,s)$ to be the maximal number such that in every hypergraph~$\mathcal{H}$ on~$n$ vertices and with at most~$ m(n,s)$ edges there is a vertex~$x$ such that~$\vert\mathcal{H}_x\vert\geq \vert E(\mathcal{H})\vert -s$, where~$\mathcal{H}_x=\{H\setminus\{x\}:H\in E(\mathcal{H})\}$.
This problem has been posed by F\"uredi and Pach and by Frankl and Tokushige. While the first results were only for specific small values of~$s$, Frankl determined~$m(n,2^{d-1}-1)$ for all~$d\in\mathds{N}$ with~$d\mid n$. Subsequently, the goal became to determine~$m(n,2^{d-1}-c)$ for larger~$c$. Frankl and Watanabe determined~$m(n,2^{d-1}-c)$ for~$c\in\{0,2\}$. Other general results were not known so far. 

Our main result sheds light on what happens further away from powers of two: We prove that~$m(n,2^{d-1}-c)=\frac{n}{d}(2^d-c)$ for~$d\geq 4c$ and~$d\mid n$ and give an example showing that this equality does not hold for $c=d$. The other line of research on this problem is to determine~$m(n,s)$ for small values of~$s$. In this line, our second result determines~$m(n,2^{d-1}-c)$ for~$c\in\{3,4\}$. This solves more instances of the problem for small~$s$ and in particular solves a conjecture by Frankl and Watanabe.

\end{abstract}

\maketitle
\section{Introduction}
A hypergraph~$\mathcal H$ is a pair~$(V, \mathcal F)$ where~$V$ is the set of vertices and~$\mathcal F\subseteq 2^V$ is the set of edges. In the literature, the problems we consider in this article are often presented in the context of families rather than hypergraphs. If not necessary, it is then not distinguished between the family~$\mathcal{F}\subseteq 2^V$ and the hypergraph~$(V,\mathcal{F})$. We will follow this notational path.

Let~$V$ be an~$n$-element set and let~$\mathcal{F}$ be a family of subsets of~$V$. For a subset~$T$ of~$V$ define the \emph{trace} of~$\mathcal{F}$ on~$T$ by~$\mathcal{F}_{|T}=\{F\cap T:F\in\mathcal{F}\}$. For integers~$n$,~$m$,~$a$, and~$b$, we write $$(n,m)\rightarrow (a,b)$$ if for every family~$\mathcal{F}\subseteq 2^V$ with~$\vert\mathcal{F}\vert \geq m$ and~$\vert V\vert=n$ there is an~$a$-element set~$T\subseteq V$ such that~$\vert \mathcal{F}_{|T}\vert\geq b$ (we also say that~$(n,m)$ \emph{arrows}~$(a,b)$).

The first type of question that was asked for this arrowing notation is similar to the spirit of the classic Tur\'an problem: For a fixed number of vertices~$n$, how many edges are needed such that there is a subset of vertices such that all its subsets lie in the trace. The following result on this question was conjectured by Erd\H{o}s~\cite{FurediPach} and was proved independently by Sauer~\cite{Sauer}, Shelah and Perles~\cite{Shelah}, and Vapnik and \v{C}ervonenkis~\cite{CervonenkisVapnik}. It states that for a large family~$\mathcal{F}$ on~$n$ vertices, there is an~$s$-set of vertices such that all its subsets lie in the trace of~$\mathcal{F}$. More precisely, they showed that~$(n,m)\rightarrow (s,2^s)\text{ whenever }m>\sum_{0\leq i<s}\binom{n}{i}$.

Another fundamental question that was raised in the area is how large a family can be at most so that there will still be a vertex~$v$ such that the trace on~$V\setminus \{v\}$ is not much smaller than the original family. More precisely, the following problem was posed by F\"uredi and Pach~\cite{FurediPach} and, more recently, by Frankl and Tokushige as Problem 3.8 in their monograph~\cite{FranklTokushige}\footnote{There have been slightly different versions in use for the arrowing notation and for what we denote by~$m(n,s)$. In this work, we follow the notation in~\cite{FranklTokushige}.}: 
\begin{problem}\label{prob:main}
    Given non-negative integers~$n$ and~$s$, what is the maximum value~$m(n,s)$ such that for every~$m\leq m(n,s)$ we have $$(n,m)\rightarrow (n-1,m-s).$$
\end{problem}

As described in the abstract, this problem can also be formulated as finding the maximal number~$m(n,s)$ such that the following holds. 
In every hypergraph~$\mathcal{H}$ with some~$n$-set~$V$ as vertex set and with at most~$m(n,s)$ edges there is a vertex~$x$ such that~$\vert\mathcal{H}_x\vert\geq \vert \mathcal{H}\vert -s$, where~$\mathcal{H}_x=\mathcal{H}_{|V\setminus\{x\}}=\left\{H\setminus\{x\}:H\in \mathcal{H}\right\}$.

A family~$\mathcal F$ is \textit{hereditary} if for every~$F'\subseteq F\in \mathcal F$ we have that~$F'\in \mathcal F$. 
In~\cite{Frankl} Frankl proves that among families with a fixed number of edges and vertices, the trace is minimised by hereditary families. 
Thus, the problems considered here, and in particular Problem~\ref{prob:main}, can be reduced to hereditary families (see Lemma~\ref{lem:heredissuf}). 
Note that in hereditary families, Problem~\ref{prob:main} is asking for the maximum number of edges such that there is always a vertex of small degree (as usual, we define the degree of a vertex~$v$ as the number of edges that contain~$v$). 

The investigation of this problem started with Bondy~\cite{Bondy} and Bollob\'as~\cite{Bollobas} determining~$m(n,0)$ and~$m(n,1)$, respectively. Later Frankl~\cite{Frankl} and Frankl and Watanabe~\cite{FranklWatanabe} proved part~(\ref{it:c=1}) and~(\ref{it:c=2}), respectively, of the following theorem.
\begin{theorem}\label{thm:genFrankl,FranklWatanabe}
    For~$d,n\in\mathds{N}$ and~$d|n$, we have
    \begin{enumerate}
        \item\label{it:c=1} $m(n,2^{d-1}-1)=\frac{n}{d}(2^d-1)$\,,
        \item\label{it:c=2} $m(n,2^{d-1}-2)=\frac{n}{d}(2^d-2)$\,.
    \end{enumerate}
\end{theorem}

Consider a family consisting of a set of size~$d$ and all possible subsets, and take~$n/d$ vertex disjoint copies of it. The resulting family has minimum degree~$2^{d-1}$ and~$\frac{n}{d}(2^d-1)+1$ edges. Thus, this family is an extremal construction for~(\ref{it:c=1}). By taking out all sets of size~$d$, we obtain an extremal construction for~(\ref{it:c=2}).

Our main result makes further progress on Problem~\ref{prob:main}, solving it for general~$s=2^{d-1}-c$ as long as~$c$ is linearly small in~$d$.
\begin{theorem}[Main theorem]\label{thm:main}
	Let~$d,c,n\in\mathds{N}$ with~$d\geq 4c$ and~$d|n$. Then $$m(n,2^{d-1}-c)=\frac{n}{d}(2^d-c) .$$
\end{theorem}

\begin{remark}
In fact, our proof of Theorem~\ref{thm:main} yields that for~$d\geq 4c$ and~$m\leq \frac{n}{d}(2^d-c)$ we have~$(n,m)\rightarrow(n-1,m-(2^{d-1}-c))$ without any divisibility conditions on~$n$.
The assumption~$d|n$ is only necessary for the extremal constructions showing the maximality of~$\frac{n}{d}(2^d-c)$.
Analogous remarks hold for Theorem~\ref{thm:genFrankl,FranklWatanabe} above and Theorem~\ref{thm:secmain} below.
In Section~\ref{sec:conrem} we provide a construction showing that the equality in Theorem~\ref{thm:main} does not hold for~$d=c$ (see Construction~\ref{constr:nonloc}).
\end{remark}
 
One might also try to solve Problem~\ref{prob:main} for small values of~$s$. Apart from the aforementioned results by Bondy and Bollob\'as, progress was made by Frankl~\cite{Frankl}, Watanabe~\cite{Watanabe, Watanabe2}, and by Frankl and Watanabe~\cite{FranklWatanabe}. In~\cite{FranklWatanabe} they conjectured that~$m(n, 12) = (28/5 + o(1))n$. Theorem~\ref{thm:main} does not consider cases for which~$d$ is very small in terms of~$c$. The following results extend Theorem~\ref{thm:genFrankl,FranklWatanabe} to~$c=3$ and~$4$ and every~$d\geq 3$ (for smaller~$d$ the respective~$m(n,s)$ is not defined). In particular, it proves the conjecture of Frankl and Watanabe for~$s=12$ in a strong sense.
\begin{theorem}\label{thm:secmain}
    Let~$d,n\in\mathds{N}$ with~$d\geq 3$ and~$d|n$. Then
    \begin{enumerate}
        \item\label{it:thm:c=3} $m(n,2^{d-1}-3)=\frac{n}{d}(2^d-3)$ and
        \item\label{it:thm:c=4} $m(n,2^{d-1}-4)=\frac{n}{d}(2^d-4)$. In particular,~$m(n,12)=\frac{28}{5}n$.
    \end{enumerate}
\end{theorem}
Note that for larger~$d$, this theorem is of course a special case of Theorem~\ref{thm:main}.

\subsection{Idea of the proof}\label{sec:ideaofproof}

To show the maximality of~$\frac{n}{d}(2^d-c)$ we give a construction similar to the one presented after Theorem~\ref{thm:genFrankl,FranklWatanabe}. 
As mentioned above, Lemma~\ref{lem:heredissuf} reduces the problem to a problem in hereditary families.
We need to show that for every hereditary hypergraph~$\mathcal F$ on~$n$ vertices with minimum degree at least~$2^{d-1}-c+1$ we have that $$\vert\mathcal F \vert \geq  \frac{n}{d}(2^d-c)+1.$$
In the proof of Theorem~\ref{thm:genFrankl,FranklWatanabe}~\cite{Frankl,FranklWatanabe} the equality~$\vert \mathcal F \setminus \{\emptyset\}\vert = \sum_{v\in V} \sum_{H\in L_v} \frac{1}{\vert H\vert +1}$, where~$L_v$ is \emph{the link of the vertex~$v$}, was used, which comes from a simple double counting argument.
Subsequently, they used a generalised form of the Kruskal-Katona Theorem (see Theorem~\ref{thm:genKK} below) to obtain a general lower bound for~$\sum_{H\in L_v}\frac{1}{\vert H\vert +1}$ for every~$v$. Due to the aforementioned double counting this in turn yields a lower bound on the number of edges.

For~$c\geq 3$ there exist extremal families which show that a general bound on~$\sum_{H\in L_v}\frac{1}{\vert H\vert +1}$ for every vertex~$v$ is not sufficient to provide the desired bound on the number of edges.
To overcome this difficulty first observe that the double counting argument can be generalised by interpreting~$\sum_{H\in L_v}\frac{1}{\vert H\vert +1}$ as the weight~$w_{\mathcal{F}}(v)$ of a vertex~$v$. 
We will refer to this weight as~\emph{uniform weight} since it can be imagined as uniformly distributing the unit weight of an edge to each of its vertices.
In contrast, to prove Theorem~\ref{thm:main} and Theorem~\ref{thm:secmain}, we will use a non-uniform weight.
Moreover, instead of bounding the weight of single vertices we will bound the weight of sets of vertices. 

To this aim take the maximum possible set of vertices with small uniform weight such that their neighbourhoods are pairwise disjoint. 
Call these vertices together with their neighbours~\emph{clusters}.
Note that in this way the neighbourhood of every vertex with small uniform weight needs to intersect some cluster. 
For bounding the weight of vertices whose neighbourhood does not intersect any cluster (and therefore have a large uniform weight), we introduce a ``local'' lemma (see Lemma~\ref{lem:weight}) which is a close relative to a general form of the Kruskal-Katona theorem. 
Given a vertex of fixed degree, it provides a lower bound on the uniform weight and furthermore, the minimum weight surplus if its link deviates enough from the minimising link.
Since the link of vertices whose neighbourhood does not intersect any cluster indeed deviates enough from the minimising link, the lemma then gives that these will have a large weight.

The next step is to bound the average weight of the vertices in each such cluster.
Even if the number of edges inside a cluster is not large enough,~$\mathcal{F}$ being hereditary and the minimum degree of~$\mathcal{F}$ still provide some lower bound for the number of edges in a cluster. 
Then a second local lemma (Lemma~\ref{lem:locfamiso}) yields that there are several vertices within that cluster whose degree (with respect to the cluster) is not the minimum degree in~$\mathcal{F}$. 
Therefore, there exist several crossing edges, i.e., edges containing vertices from both the inside and the outside of the cluster.
If we use the uniform weight, these crossing edges will contribute enough to the weight of the cluster, even more than needed.

At this point we still need to bound the weight of vertices with small uniform weight lying outside of any cluster.
As mentioned above, the neighbourhood of every such vertex intersects some cluster, meaning every such vertex is contained in a crossing edge.
Recall that in fact, a uniform weight on crossing edges would contribute more weight than needed for the inside of a cluster. 
Now the second idea comes into play: the unit weight of these edges will be distributed non-uniformly among its vertices.
Hence, when splitting the unit weight of such a crossing edge according to the above mentioned imbalance, both sides will get a share that is big enough.

Note that this strategy is compatible with the extremal constructions in so far as that those are composed of disjoint copies of locally optimal families.

\section{Preliminaries}

In this work we consider the set of natural numbers~$\mathds{N}$ to start with~$1$ and the logarithms considered are to the base~$2$. Further, for~$i\in\mathds{Z}$ we set as usual~$[i]=\{1,\dots ,i\}$, and it is also convenient to define~$[i]_0=\{0,\dots ,i\}$. Given a set~$F\subseteq \mathds{N}$ and some~$i\in\mathds{N}$, we denote by~$F+i$ the set~$\{j+i:j\in F\}$. For our considerations isolated vertices, i.e., vertices that are contained in the vertex set of a hypergraph but do not lie in any edges, usually do not play an important r\^ole. This will lead to a few easy peculiarities in notation.
For two hypergraphs~$\mathcal{H}$ and~$\mathcal{H}'$ we write~$\mathcal{H}\cong \mathcal{H}'$ if they are 
isomorphic up to isolated vertices,
more precisely, if there are vertex sets~$V$ disjoint to~$V(\mathcal{H})$ and~$V'$ disjoint to~$V(\mathcal{H}')$ such that the hypergraph~$(V(\mathcal{H})\dcup V,E(\mathcal{H}))$ is isomorphic to~$(V(\mathcal{H}')\dcup V',E(\mathcal{H}'))$.

For a hypergraph~$\mathcal{H}=(V,E)$ and~$v\in V$ we define the link~$L_v$ of~$v$ to be the hypergraph on~$V$ with edge set~$\{F\setminus\{v\}: F\in E\}$.
Further, we write $$V_v=\{w\in V\colon\text{there is an } e\text{ with }\{v,w\}\subseteq e\in E\} \, ,$$ note that if~$v$ is not an isolated vertex, then~$v\in V_v$. This notation will be useful in the proof of Theorem~\ref{thm:main} when defining the clusters mentioned in the overview of the proof.



The following lemma due to Frankl~\cite{Frankl} provides the aforementioned reduction of Problem~\ref{prob:main} to hereditary families.
\begin{lemma}\label{lem:heredissuf}
	For~$n,m,a,b\in\mathds{N}$ the following statements are equivalent.
	\begin{enumerate}
		\item For every~$n$-set~$V$ and every hereditary family~$\mathcal{F}\subseteq 2^V$ with~$\vert \mathcal{F}\vert \geq m$, there exists a set~$T\subseteq V$ with~$\vert T\vert = a$ such that~$\vert\mathcal{F}_{|T}\vert\geq b$.
		\item $(n,m)\rightarrow(a,b)$
	\end{enumerate}
\end{lemma}

In particular, this means that in the proof of our results we only need to consider hereditary families.

Let~$n\in\mathds{N}$, for~$A,B\in 2^{[n]}$ we say that~$A\prec _{col}B$ or~$A$ precedes~$B$ in the \emph{colexicographic order} if~$\max (A\vartriangle B) \in B$. 
Let~$m\in\mathds{N}$ with~$m\leq 2^n$ and define~$\mathcal{R}_n(m)$ to be the family on~$n$ vertices containing the first~$m$ sets of~$2^{[n]}$ according to the colexicographic order. 
Note that for~$n\leq n'$ and~$m\leq 2^n$, we have~$\mathcal{R}_n(m)\cong \mathcal{R}_{n'}(m)$ and hence, we will not distinguish between~$\mathcal{R}_n(m)$ and~$\mathcal{R}_{n'}(m)$ and we will omit the subscript. 
The following theorem due to Katona~\cite{Katona} is a generalisation of the well-known Kruskal-Katona theorem.
\begin{theorem}\label{thm:genKK}
Let~$f: \mathds{N}_0 \rightarrow \mathds{R}$ be a monotone non-increasing function and let~$\mathcal{F}$ be a hereditary family with~$\vert \mathcal{F} \vert = m$. Then $$\sum_{F\in\mathcal{F}}f(\vert F\vert )\geq \sum_{R\in\mathcal{R}(m)}f(\vert R\vert ) .$$
\end{theorem}

In the proofs of Theorems~\ref{thm:main} and~\ref{thm:secmain}, normally Theorem~\ref{thm:genKK} is applied with $\mathcal F$ being the link of a vertex. 
Moreover, as we usually consider the uniform weight mentioned in Section~\ref{sec:ideaofproof}, the function~$f$ will often be~$f(k)=\frac{1}{k+1}$.
The weight of~$\mathcal{R}(m)$ with respect to this~$f$ will come up repeatedly and hence, for brevity we set~$W(m):=\sum_{R\in \mathcal{R}(m)}\frac{1}{\vert R\vert +1}$. Note that we have~$W(2^{d-1})=\frac{2^d-1}{d}$ and further the following estimate\footnote{\label{fn:bounds}To have a clearer presentation of our main results and their proofs, we refrained from striving for optimal bounds.} for~$W(2^{d-1}-c)$ for a~$c\in [2^{d-2}]$:
\begin{align}\label{eq:boundW(2^d-c)}
    W(2^{d-1}-c)\geq \frac{2^{d}-1}{d}-\frac{c}{d-\log c}
\end{align}

Indeed, if~$A\in 2^{[d-1]}\setminus\mathcal{R}(2^{d-1}-c)$, then there are at least~$2^{d-1-\vert A\vert}$ sets in~$2^{[d-1]}\setminus\mathcal{R}(2^{d-1}-c)$. 
Thus, it follows that for every~$A\in 2^{[d-1]}\setminus\mathcal{R}(2^{d-1}-c)$ we have~$\vert A\vert\geq d-1-\log c$. 
This gives that~$W(2^{d-1})-W(2^{d-1}-c)\leq \frac{c}{d-\log c}$ and thereby~\eqref{eq:boundW(2^d-c)}.

\section{Proof of Theorem~\ref{thm:main}}

As mentioned in Section~\ref{sec:ideaofproof} for proving Theorem~\ref{thm:main} we introduce two ``local'' lemmas. 
The first lemma says that if a family deviates enough from~$\mathcal{R}(m)$, the weight of this family will have a surplus with respect to~$W(m)$\textsuperscript{\ref{fn:bounds}}.

\begin{lemma}\label{lem:weight}
	Let~$d\geq 4$ and~$c\leq 2^d$ be integers. For a hereditary family~$\mathcal H$, with~$\vert\mathcal H \vert\geq 2^d-c$ the following holds. 
	\begin{enumerate}
	\item\label{it:weightineq}$\sum_{H\in\mathcal{H}}\frac{1}{\vert H\vert +1}\geq W(2^d-c) $. 
    \item\label{it:surplusmorevert} If there are at least~$d+1$ non isolated vertices in~$\mathcal{H}$, then
	\begin{align*}
	\sum_{H\in\mathcal{H}}\frac{1}{\vert H\vert +1}\geq
        W(2^d-c) + \frac{1}{6} \,.
    \end{align*}
    \item\label{it:surplusspec}If~$c\in \{2,3\}$ and~$\mathcal{H}\not\cong \mathcal{R}(2^d-c)$, then we have
	\begin{align*}
	\sum_{H\in\mathcal{H}}\frac{1}{\vert H\vert +1}\geq W(2^d-c)+\min\left(\frac{1}{6},\frac{1}{d}\right)\,.
    \end{align*}
    \end{enumerate}
\end{lemma}

\begin{proof}
	Let~$d$,~$n$,~$c$, and~$\mathcal{H}$ be given as in the statement. The first part follows by applying Theorem~\ref{thm:genKK} with~$f(k)=\frac{1}{k+1}$.
	
	In order to prove part~(\ref{it:surplusmorevert}) and~(\ref{it:surplusspec}) we need some preparation.
	Denote by~$h_i$ and~$r_i$ the number of~$i$-sets in~$\mathcal{H}$ and~$\mathcal{R}(2^d-c)$, respectively. 
	Given~$s\in [d]_0$ set~$g(k)=1$ for~$k\leq s$ and~$g(k)=0$ for~$k>s$. 
	Then Theorem~\ref{thm:genKK} applied with~$f=g$ yields
	\begin{align}\label{eq:numsetsinlayer}
	    \sum_{i\in [s]_0}h_i\geq\sum_{i\in [s]_0}r_i .
	\end{align}
	
	Next, let~$H_1,\dots ,H_{\vert\mathcal{H}\vert}$ be an enumeration of the elements of~$\mathcal{H}$ such that~$\vert H_j\vert\leq \vert H_{j+1}\vert$. 
	Given~$i\in [d-1]$ let~$\varphi(i)$ be the number of edges of size at most $i$ in the family~$\mathcal{R}(2^d-c)$, i.e.,~$\varphi(i)=\sum_{j\in [i]_0} r_j$.
	Let~$\mathcal{H}_0=\{H_1\}=\{\emptyset\}$ and for~$i\in [d-1]$ consider the following set of edges~$\mathcal H_i=\{H_{\varphi(i-1) +1}, \dots, H_{\varphi(i)}\}$ and observe that its size is $r_i$. 
	Inequality (\ref{eq:numsetsinlayer}) implies that for~$H\in\mathcal{H}_i$, where~$i\in [d-1]_0$, we have~$\vert H\vert\leq i$. 
	Thus, 
	\begin{align}\label{eq:weightaswanted}
	\sum_{i\in [d-1]_0}\sum_{H\in\mathcal{H}_i}\frac{1}{\vert H\vert +1}\geq\sum_{i\in [d-1]_0}\frac{r_i}{i+1}=W(2^d-c) \, .
	\end{align}
	If now at least~$d+1$ vertices are contained in edges of~$\mathcal{H}$, then even for~$H_{d+2} \in \mathcal H_2$ it holds that~$\vert H_{d+2}\vert =1$.
	Hence,~(\ref{eq:weightaswanted}) now becomes~$\sum_{H\in\mathcal{H}}\frac{1}{\vert H\vert+1}\geq\frac{1}{2}-\frac{1}{3}+W(2^d-c)$ and (\ref{it:surplusmorevert}) is proved.
	
	For proving~\eqref{it:surplusspec}, let~$c\in\{2,3\}$ and note that if there are at least~$d+1$ non isolated vertices in~$\mathcal H$, then the result follows from~\eqref{it:surplusmorevert}. Thus, assume that there are only~$d$ non isolated vertices in~$\mathcal{H}$.
	Observe that~$r_i=\binom{d}{i}$ for~$i\in [d-2]$,~$r_d=0$ and~$r_{d-1}=d-(c-1)$.
	Hence, due to~\eqref{eq:numsetsinlayer} we have~$h_i=\binom{d}{i}$ for~$i\in [d-2]$ and because of~$\mathcal{H}$ being hereditary and the size of~$\mathcal{H}$, further~$h_{d-1}\geq d-(c-1)$. 
	In fact,~$h_{d-1}>r_{d-1}=d-(c-1)$ has to hold since~$\mathcal{H}\not\cong \mathcal{R}(2^d-c)$. 
	Together with~(\ref{eq:weightaswanted}) the result follows.
\end{proof}

The following is the second local lemma mentioned in the overview. 
Part~\eqref{it:locdeg} states that a family on~$d$ vertices with high minimum degree contains many edges and therefore, considering Lemma~\ref{lem:heredissuf}, this is a local version of Theorem~\ref{thm:main}.
Moreover, Part~\eqref{it:locnumedges} states that if a family has not enough edges, then there are several vertices of low degree.

\begin{lemma}\label{lem:locfamiso}
    Let~$d,c\in\mathds{N}$,~$V$ be a~$d$-set and let~$\mathcal{H}\subseteq 2^V$ be hereditary.
    \begin{enumerate}
        \item\label{it:locnumedges} If~$\vert\mathcal{H}\vert\leq 2^d-c-1$, then~$\deg(v)\leq 2^{d-1}-c-1$ for at least~$d-c$ vertices~$v$.
        \item\label{it:locdeg} If~$d\geq c+1$ and~$\delta(\mathcal{H})\geq 2^{d-1}-c$, then~$\vert\mathcal{H}\vert\geq 2^d-c$.
    \end{enumerate}
\end{lemma}

\begin{proof}
    
    (\ref{it:locnumedges}): By~$\widebar{\mathcal{H}}$ denote the family~$\{V\setminus F: F\in 2^{V}\setminus\mathcal{H}\}$. 
	The bound on~$\vert \mathcal{H}\vert$ implies that~$c+1\leq\vert\widebar{\mathcal{H}}\vert$ and observe that since~$\mathcal{H}$ is hereditary,~$\widebar{\mathcal{H}}$ is hereditary.
	Consider some ordering~$\widebar{\mathcal{H}}=\{H_1,\dots ,H_{\vert \widebar{\mathcal{H}}\vert}\}$ with~$\vert H_i\vert\leq\vert H_{i+1}\vert$. 
	Note that because~$\widebar{\mathcal{H}}$ is hereditary, we know that if some vertex~$v\in V$ is contained in one of the edges~$H_1,\dots ,H_{j}$, then in fact~$\{v\}=H_i$ for some~$i\in [j]$. 
	Thus, there are~$d-c$ vertices that do not lie in any of~$H_1,\dots ,H_{c+1}$. Note that these vertices lie in at least~$c+1$ sets of~$2^V\setminus \mathcal{H}$ and therefore, for each such~$v$ we have~$\deg_{\mathcal{H}}(v)\leq 2^{d-1}-c-1$.
	
	(\ref{it:locdeg}): Assume for contradiction that~$\vert \mathcal{H}\vert\leq 2^d-c-1$. Then (\ref{it:locnumedges}) gives the contradiction. 
\end{proof}

Now we are ready to prove Theorem~\ref{thm:main}.

\begin{proof}[Proof of Theorem~\ref{thm:main}]
    Let~$n$,~$d$, and~$c$ be given as in the theorem. 
    First note that $$\mathcal{F}_0=\left\{F+(i-1): F\in\mathcal{R}(2^{d}-(c-1)) \text{ and }i\in \left[\frac{n}{d}\right]\right\}\subseteq 2^{[n]}$$ shows that for~$m=\frac{2^d-c}{d}n+1$, we have~$\left(n,m\right)\not\rightarrow \left( n-1, m-(2^{d-1}-c)\right)$.
	
	In an hereditary family on~$n$ vertices with~$m$ edges the existence of a set of size~$n-1$ on which the trace of the family has size at least~$m-(2^{d-1}-c)$ is equivalent to the existence of a vertex with degree at most~$2^{d-1}-c$.
	Therefore, Lemma~\ref{lem:heredissuf} implies that it is sufficient to show that for every hereditary family~$\mathcal{F}$ on~$n$ vertices with minimum degree at least~$2^{d-1}-c+1$ we have~$\vert\mathcal{F}\vert\geq\frac{2^d-c}{d}n+1$.
	Let now~$\mathcal{F}\subseteq 2^V$ be such a hereditary family on some~$n$-set~$V$ in which every vertex has degree at least~$2^{d-1}-c+1$. 
	
	To prove the lower bound on the number of edges, we will define a weight function~$w$ on~$V$ with the property that~$1+\sum_{v\in V}w_{\mathcal{F}}(v)\leq \vert\mathcal{F}\vert$. 
	Subsequently, it will be enough to show that~$\sum_{v\in V}w_{\mathcal{F}}(v)\geq \frac{2^d-c}{d}n$. 
	Indeed, for~$c=1$ the weight function~$\sum_{H\in L_v}\frac{1}{\vert H\vert +1}$ together with Lemma~\ref{lem:weight} provide this, so from now on we assume~$c\geq 2$. 
	Note however, that for this uniform weight and~$c$ large, in~$\mathcal{F}_0$ there are vertices with weight below and above~$\frac{2^d-c}{d}$.
	As mentioned in the overview, we overcome this difficulty by using non-uniform weights and by bounding the average weight of sets of vertices instead of bounding the weight of every single vertex.
    
    To that aim, we will in the following consider a partition of~$V$. 
	Let us call a vertex~$v\in V$ \textit{light} if~$\vert V_v\vert=d$. 
	Further, let~$\mathcal{L}$ be a maximum set of light vertices such that~$V_v\cap V_{v'}=\emptyset$ for all~$v,v'\in\mathcal{L}$ and call the sets~$V_v$ with~$v\in \mathcal{L}$ \emph{clusters}. 
	Later, the weight of a vertex will be defined depending on how it relates to to these clusters. 
	Moreover, call the vertices~$u\in V\setminus\bigcup_{v\in \mathcal{L}}V_v$ with~$\vert V_u\vert >d$ \textit{heavy} vertices and let~$\mathfrak{H}$ be the set of all heavy vertices. 
	The vertices in~$\mathcal{L}$ will be distinguished further into two different types~$\mathcal{L}_1$ and~$\mathcal{L}_2$ as follows. 
	Let~$\mathcal{L}_1$ be the set of those vertices~$v\in\mathcal{L}$ for which every vertex in~$V_v$ is only contained in edges of~$2^{V_v}$, 
	that is
    $$\mathcal L_1 = \{v\in \mathcal L\colon \text{there is no }e\in \mathcal F\setminus 2^{V_v}  \text{ with  }e\cap V_v\neq \emptyset\}. $$
	Furthermore, let~$\mathcal{L}_2$ be the set of those vertices~$v\in\mathcal{L}$ for which there exists an~$x\in V_v$ that is contained in an edge of~$\mathcal{F}\setminus 2^{V_v}$, in other words,
	\begin{align}\label{def:lighttypes}
	    \mathcal L_2 &= \{v\in \mathcal L\colon \text{there is an }e\in \mathcal F\setminus 2^{V_v}  \text{ with  }e\cap V_v\neq \emptyset \}.
	\end{align}

	Note that we have~$\mathcal{L}=\mathcal{L}_1\dcup\mathcal{L}_2$. Lastly, we collect the remaining vertices in the set~$\widebar{\mathcal{L}}=V\setminus (\mathfrak{H}\cup\bigcup_{v\in\mathcal{L}}V_v)$. 
	Thus, we have~$V=\mathfrak{H}\dcup\bigcup_{v\in\mathcal{L}_1} V_v\dcup\bigcup_{v\in\mathcal{L}_2} V_v \dcup \widebar{\mathcal{L}}$.
    
    Next, for each of the partition classes~$\mathfrak{H}$,~$\bigcup_{v\in\mathcal{L}_1}V_v$,~$\bigcup_{v\in \mathcal{L}_2}V_v$, and~$\widebar{\mathcal{L}}$ the weights will be defined and we will show that the average weight in each partition class is bounded from below by~$\frac{2^d-c}{d}$.
    
	Assign the uniform weight~$w_{\mathcal{F}}(u)=\sum_{H\in L_u}\frac{1}{\vert H\vert +1}$ to every heavy vertex~$u\in\mathfrak{H}$. 
	This definition and (\ref{it:surplusmorevert}) from Lemma~\ref{lem:weight} give that every heavy vertex has weight at least 
	\begin{align}\label{eq:weightestimateheavy0}
	\frac{1}{6}+W(2^{d-1}-c+1)\geq \frac{1}{6}+\frac{2^d-1}{d}-\frac{c-1}{d-\log (c-1)}\geq \frac{2^d-c}{d} \, ,
	\end{align}
	where we used the bound (\ref{eq:boundW(2^d-c)}) for the first inequality and~$d\geq 4c$ and~$\log x\leq \frac{2}{3}x$ for~$x\geq 1$ for the second (recall that we can assume~$c\geq 2$).
	
	Given~$v\in\mathcal{L}_1$, we have that~$\mathcal{F}[V_v]$ is a family on~$d$ vertices with minimum degree at least~$2^{d-1}-c+1$. 
	Thus, from Lemma~\ref{lem:locfamiso}~(\ref{it:locdeg}) (with~$c-1$ here in place of~$c$ there) it follows that~$\vert\mathcal{F}[V_v]\vert\geq 2^d-c+1$. 
	Since summing the uniform vertex weights of all vertices of a family amounts to the number of non-empty edges in that family, 
	assigning the uniform weight~$w_{\mathcal{F}}(x)=\sum_{H\in L_x}\frac{1}{\vert H\vert +1}$ to every~$x\in V_v$ yields
	\begin{align}\label{eq:weightestimatetype10}
	\frac{1}{d}\sum_{x\in V_v}w_{\mathcal{F}}(x)= \frac{\vert\mathcal{F}[V_v]\setminus\{\emptyset\}\vert}{d}\geq \frac{2^d-c}{d} \, .
	\end{align}
    
    Given~$v\in\mathcal{L}_2$, the idea is that the vertices in~$V_v$ already have a relatively large uniform weight just taking into account the edges on~$V_v$. 
	Thus, they only need a smaller proportion of the weight of an edge that includes vertices outside of~$V_v$. 
	More precisely, we assign the weight 
	$$w_{\mathcal{F}}(x)=\sum_{H\in L_x}\frac{1}{\vert H\vert +1}-\left\vert V_x\setminus V_v\right\vert\left(\frac{1}{2}-\frac{c-1}{d-c}\right)$$
	to every vertex~$x\in V_v$. 
	This definition can be understood as vertices in~$V_v$ basically having the uniform weight but then renouncing part of their uniform share of $2$-uniform edges that cross from the inside of a cluster to the outside.
	Later, these crossing edges will contribute more than their uniform share to the outside vertex.
	
	Of course, if~$\vert\mathcal{F}[V_v]\vert\geq 2^d-c+1$, then again the bound~(\ref{eq:weightestimatetype10}) follows for~$v$ directly by double counting and thus, we may assume that~$\vert\mathcal{F}[V_v]\vert\leq 2^d-c$. 
	Define the set~$C$ as the set of vertices~$x\in V_v$ for which there exists some~$F_x$ with~$x\in F_x\in\mathcal{F}\setminus 2^{V_v}$.
	Note that in fact, since~$\mathcal{F}$ is hereditary, we may assume~$\vert F_x\vert =2$. 
	Considering the minimum degree condition in~$\mathcal F$ and applying Lemma~\ref{lem:locfamiso}~(\ref{it:locnumedges}) to~$\mathcal{F}[V_v]$ (with~$c-1$ here instead of~$c$ there) it follows that 
	\begin{align}\label{eq:Cbound}
	    \vert C\vert\geq d-c+1 \, .
	\end{align}
	Moreover, the minimum degree of~$\mathcal{F}$ implies~$d(v)\geq 2^{d-1}-c+1$ and hence,~$\mathcal{F}$ being hereditary gives that~$\vert 2^{V_v}\setminus\mathcal{F}\vert\leq 2(c-1)$. 
	Therefore, double counting the non-empty edges in~$\mathcal{F}[V_v]$ yields
	\begin{align}\label{eq:weightestimatetype2inner0}
	\vert \mathcal{F}[V_v]\setminus\{\emptyset\}\vert=
	\sum_{x\in V_v}\sum_{H\in L_x\cap 2^{V_v}}\frac{1}{\vert H\vert +1}\geq 2^d-2c+1 \, .
	\end{align}
	Now, observe that the definition of the weight together with~\eqref{eq:weightestimatetype2inner0} and~\eqref{eq:Cbound} give
	\begin{align}\label{eq:weightestimatetype20}
	\frac{1}{d}\sum_{x\in V_v}w_{\mathcal{F}}(x) 
	&\geq 
	\frac{1}{d}\left(\vert \mathcal F[V_v] \setminus \{\emptyset\}\vert+\vert  C\vert\frac{c-1}{d-c}\right) \nonumber \\
	&\geq
	\frac{1}{d}\left(2^d-2c+1+(d-c+1)\frac{c-1}{d-c}\right)  \nonumber\\
	&\geq 
	\frac{2^d-c}{d}\, . 
	\end{align}
    
    Lastly consider vertices from~$\widebar{\mathcal{L}}$. 
	Recall that in particular, these vertices are light and could potentially have a too low weight if the uniform weight would be used. 
	Note that by the maximality of~$\mathcal{L}$, for every vertex~$a\in\widebar{\mathcal{L}}$ we can pick a~$v(a)\in\mathcal{L}_2$ such that there exists an edge containing~$a$ and a vertex of~$V_{v(a)}$. 
	Since the vertices in~$\bigcup_{v\in\mathcal{L}_2}V_v$ renounced their full share of some of those edges, the vertices in~$\widebar{\mathcal{L}}$ can be given a larger fraction. 
	To be precise, the weight for~$a\in\widebar{\mathcal{L}}$ is defined as
	$$w_{\mathcal{F}}(a)=\sum_{H\in L_a}\frac{1}{\vert H\vert +1}+\left\vert V_a\cap V_{v(a)}\right\vert \left(\frac{1}{2}-\frac{c-1}{d-c}\right) \, .$$
	Lemma~\ref{lem:weight} (\ref{it:weightineq}) yields that
	\begin{align}\label{eq:weightestimatetype30}
	w_{\mathcal{F}}(a)\geq W(2^{d-1}-c+1)+\frac{1}{2}-\frac{c-1}{d-c}\geq W(2^{d-1}-c+1)+\frac{1}{6}\geq \frac{2^d-c}{d} \, ,
	\end{align}
	where the second inequality follows from~$d\geq 4c$ and the third follows as in (\ref{eq:weightestimateheavy0}).
	Observe that the definition of~$w_{\mathcal{F}}$ implies~$\sum_{x\in V}w_{\mathcal{F}}(x)\leq 1+ \vert\mathcal{F}\vert$ because the left-hand side counts every edge of~$\mathcal{F}$ apart from the empty set at most once. 
	Since (\ref{eq:weightestimateheavy0}), (\ref{eq:weightestimatetype10}), (\ref{eq:weightestimatetype20}), and (\ref{eq:weightestimatetype30}) say that the average weight per vertex in~$\mathcal{F}$ is at least~$\frac{2^d-c}{d}$, the proof is complete. \end{proof}

\section{Proof of Theorem~\ref{thm:secmain}}

This section is dedicated to the proof of Theorem~\ref{thm:secmain}.
The proof is very similar to the proof of the main theorem just with some adaptions to obtain more precise bounds at certain points.
Hence, we will omit some details that already appeared in the last section. 

\begin{proof}[Proof of Theorem~\ref{thm:secmain}]
    
    Let~$c\in\{3,4\}$ and note that the cases~$d=3$ and~$4$ have been solved before, see~\cite{FranklWatanabe} and~\cite{Watanabe}, so assume~$d\geq 5$.
    Firstly, the family~$\mathcal{F}_0$ from the proof of Theorem~\ref{thm:main} shows that for~$m=\frac{n}{d}(2^{d}-c)+1$, we have~$\left(n,m\right)\not\rightarrow \left( n-1, m-(2^{d-1}-c)\right)$.

    Let now~$\mathcal{F}\subseteq 2^V$ be a hereditary family on some~$n$-set~$V$ in which every vertex has degree at least~$2^{d-1}-c+1$.
    In the following we will show that~$\vert \mathcal{F}\vert\geq \left(2^d-c\right)\frac{n}{d}+1$.
    
    To gain more precision later, this time we call a vertex~$v\in V$ \textit{light} if~$L_v\cong \mathcal{R}(2^{d-1}-(c-1))$. 
    Again, let~$\mathcal{L}$ be a maximum set of light vertices such that~$V_v\cap V_{v'}=\emptyset$ for all~$v,v'\in\mathcal{L}$. 
    Call the vertices~$u\in V\setminus\bigcup_{v\in \mathcal{L}}V_v$ with~$L_u\not\cong \mathcal{R}(2^{d-1}-(c-1))$ \textit{heavy} vertices. 
    The sets~$\mathcal{L}_i$,~$\mathfrak{H}$,~$\widebar{\mathcal{L}}$ are defined similarly as in the proof of Theorem~\ref{thm:main}, just according to the different definitions of light and heavy vertices here.
    
    Again we assign the uniform weight to every heavy vertex of~$\mathcal{F}$. 
    Note that then, due to Lemma~\ref{lem:weight}~(\ref{it:surplusspec}) and the structure of~$\mathcal{R}(2^{d-1}-c)$ for~$c\leq 4$, every heavy vertex has weight at least 
    \begin{align}\label{eq:weightestimateheavy1}
        \min\left(\frac{1}{6},\frac{1}{d}\right)+W(2^{d-1}-(c-1))\geq \min\left(\frac{1}{6},\frac{1}{d}\right)+\frac{2^d-1}{d}-\frac{(c-1)d-1}{(d-1)d}\geq \frac{2^d-c}{d} \, .
    \end{align} 
    
    For~$v\in\mathcal{L}_1$ and~$x\in V_v$ the weight is again defined as the uniform weight and as in the proof of Theorem~\ref{thm:main}, we obtain
    \begin{align}\label{eq:weightestimatetype11}
        \frac{1}{d}\sum_{x\in V_v}w_{\mathcal{F}}(x)\geq \frac{2^d-c}{d} \, .
    \end{align}
    
    To write the next weight definitions in a compact way, we define the following set 
    \begin{align*}
        \mathcal{S}&=\big\{H\in\mathcal{F}:\vert H\vert =3\text{ and }H\cap\bigcup_{v\in \mathcal{L}_2}V_v,H\cap\widebar{\mathcal{L}}\neq\emptyset\big\}
    \end{align*}
    
    Note that~$\mathcal{S}$ is the set of those edges of size~$3$ in~$\mathcal{F}$ crossing from the inside of some~$V_v$ with~$v\in\mathcal{L}_2$ to its outside and contain a vertex from~$\widebar{\mathcal{L}}$.
    For~$v\in\mathcal{L}_2$ and a vertex~$x\in V_v$, assign the weight~$w_{\mathcal{F}}(x)=\sum_{H\in L_x}\frac{1}{\vert H\vert +1}-\frac{1}{9}\vert \{H\in L_x:H\cup\{x\}\in\mathcal{S}\}\vert$. 
     
    \begin{claim}\label{cl:weightL2}
    	For~$v\in\mathcal{L}_2$ we have~$\frac{1}{d}\sum_{x\in V_v}w_{\mathcal{F}}(x)\geq\frac{2^d-c}{d}$.
    \end{claim}
	We postpone the proof of this claim to the end of the section and first finish the proof of Theorem~\ref{thm:secmain} using the claim.
	
	For a vertex define the weight~$a\in\widebar{\mathcal{L}}$ as~$w_{\mathcal{F}}(a)=\sum_{H\in L_a}\frac{1}{\vert H\vert +1}+\frac{1}{18}\vert \{H\in L_a:H\cup\{a\}\in\mathcal{S}\}\vert \, .$ 
	Note that by the maximality of~$\mathcal{L}$, there exists a~$v(a)\in\mathcal{L}_2$ such that there are an edge~$F$ and a vertex~$x_a\in V_{v(a)}$ with $a,x_a\in F$.
	In fact, it is easy to check that since~$L_a\cong\mathcal{R}(2^{d-1}-(c-1))$, the number of~$2$-sets in~$L_a$ that contain~$x_a$ is at least~$d-2\geq 3$. 
	Thus, Lemma~\ref{lem:weight} (\ref{it:weightineq}) and the definition of the weight yield
	\begin{align}\label{eq:weightestimatetype31}
		w_{\mathcal{F}}(a)\geq W(2^d-(c-1))+\frac{d-2}{18}\geq
		W(2^d-(c-1))+\frac{1}{6} \geq
		\frac{2^d-c}{d} \, .
	\end{align}
	
	Now observe that the definition of~$w_{\mathcal{F}}$ implies~$\sum_{x\in V}w_{\mathcal{F}}(x)\leq 1+ \vert\mathcal{F}\vert$ because the left-hand side counts every edge of~$\mathcal{F}$ apart from the empty set at most once. 
	In particular, for~$H\in\mathcal{S}$ there are at least one~$x\in H\cap\bigcup_{v\in\mathcal{L}_2}V_v$ and at most two~$a,a'\in H\cap\widebar{\mathcal{L}}$.
	Thus,~$H$ contributes at most~$1$ to~$\sum_{x\in V}w_{\mathcal{F}}(x)$.
	
	Since~(\ref{eq:weightestimateheavy1}),~(\ref{eq:weightestimatetype11}), Claim~\ref{cl:weightL2}, and~(\ref{eq:weightestimatetype31}) say that the average weight per vertex in~$\mathcal{F}$ is at least~$\frac{2^d-c}{d}$, the proof is complete.
\end{proof}

\begin{proof}[Proof of Claim~\ref{cl:weightL2}]
	Here, we will differ slightly depending on the value of~$c$.
	
	Case~$c=3$: If~$\delta (\mathcal{F}[V_v])\geq 2^{d-1}-2$, then~(\ref{eq:weightestimatetype11}) holds for~$v$ as well and so we may assume~$\delta (\mathcal{F}[V_v])< 2^{d-1}-2$ and thereby~$\vert 2^{V_v}\setminus\mathcal{F}\vert\geq 3$. 
	On the other hand, since~$d(v)\geq 2^{d-1}-2$ and~$\mathcal{F}$ is hereditary,~$\vert 2^{V_v}\setminus\mathcal{F}\vert\leq 4$. 
	So we can assume that~$\vert 2^{V_v}\setminus\mathcal{F}\vert\in \{3,4\}$. 
	If~$\vert 2^{V_v}\setminus\mathcal{F}\vert =3$, then~$d(v)\geq 2^{d-1}-2$ and~$\mathcal{F}$ being hereditary imply that the sets in~$2^{V_v}\setminus\mathcal{F}$ are~$V_v$,~$V_v\setminus \{v\}$, and some~$A\in (V_v)^{(d-1)}$ with~$v\in A$. 
	Thus, each vertex~$x\in A\smallsetminus\{v\}$ lies in all three sets of~$2^{V_v}\setminus \mathcal{F}$, and so there has to be an~$F_x\in L_x\cap (V\setminus V_v)^{(1)}$ because of the minimum degree of~$\mathcal{F}$. 
	Thus, the definition of the weight and double counting the non-empty edges in~$\mathcal{F}[V_v]$ implies
	$$ \sum_{x\in V_v}w_{\mathcal{F}}(x)\geq  \vert\mathcal F [V_v]\smallsetminus\{\emptyset\}\vert + \frac{\vert A\smallsetminus\{v\}\vert}{2} \geq   2^d-4+\frac{d-2}{2}\geq 2^d-3.$$
	
	Similarly, if~$\vert 2^{V_v}\setminus\mathcal{F}\vert =4$, then the sets in~$2^{V_v}\setminus\mathcal{F}$ are~$V_v$,~$V_v\setminus \{v\}$, some~$A\in (V_v)^{(d-1)}$ with~$v\in A$, and~$A\setminus\{v\}$. 
	Hence, there are~$d-2$ vertices~$x$ (namely, the vertices in~$A\setminus\{v\}$) for which there has to be an~$F_x\in L_x\cap (V\setminus V_v)^{(1)}$ and at least one further~$F_x'\in L_x$ with~$F_x'\cap (V\setminus V_v)\neq\emptyset$ and~$\vert F_x'\vert\leq 2$.
	Noting that each~$F_x$ contributes~$1/2$ to~$\sum_{x\in V_v}w_{\mathcal{F}}(x)$ and each~$F_x'$ at least~$1/3-1/9=2/9$, we obtain in the usual way 
	$$ \sum_{x\in V_v}w_{\mathcal{F}}(x)\geq  2^d-5+\frac{d-2}{2}+\frac{2(d-2)}{9}\geq 2^d-3  $$ and thereby the claim if~$c=3$.
	
	Case~$c=4$: In a similar way as in the beginning of the case~$c=3$, we observe that we may assume~$\vert 2^{V_v}\setminus\mathcal{F}\vert\in \{4,5,6\}$. 
	Further observe that if~$\vert 2^{V_v}\setminus\mathcal{F}\vert =4$, then since~$d(v)=2^{d-1}-3$, the sets in~$2^{V_v}\setminus\mathcal{F}$ are~$V_v$,~$V_v\setminus\{v\}$,~$A$, and~$B$ for some distinct~$A,B\in V_v^{(d-1)}$ which both contain~$v$. 
	Thus, there are at least~$d-3$ vertices (namely those in~$A\cap B\setminus\{v\}$) that lie in four sets of~$2^{V_v}\setminus \mathcal{F}$. 
	Since for any such vertex~$x$ there has to be an~$F_x\in L_x\cap (V\setminus V_v)^{(1)}$, we get~$\sum_{x\in V_v}w_{\mathcal{F}}(x)\geq 2^d-5+\frac{d-3}{2}\geq 2^d-4$. 
	
	Similarly, if~$\vert 2^{V_v}\setminus\mathcal{F}\vert =5$, the sets in~$2^{V_v}\setminus\mathcal{F}$ are~$V_v$,~$V_v\setminus\{v\}$,~$A$,~$B$,~and~$A\setminus\{v\}$ for some distinct~$A,B\in V_v^{(d-1)}$ which both contain~$v$. 
	Hence, for the~$d-3$ vertices~$x\in A\cap B\smallsetminus\{v\}$ there have to be an~$F_x\in L_x\cap (V\setminus V_v)^{(1)}$ and at least one further~$F_x'\in L_x$ with~$F_x'\cap (V\setminus V_v)\neq\emptyset$ and~$\vert F_x'\vert\leq 2$. 
	In addition, for the one vertex~$x\in A\setminus B$ there has to be an~$F_x\in L_x\cap (V\setminus V_v)^{(1)}$.
	For a vertex~$x\in A\cap B\setminus\{v\}$ we observe the following. If~$F_x'\not\in\mathcal{S}$, then~$F_x'$ contributes at least~$1/3$ to~$\sum_{x\in V_v}w_{\mathcal{F}}(x)$. 
	On the other hand, if~$F_x'\in\mathcal{S}$, then there is some~$a\in\widebar{\mathcal{L}}$ with~$a\in F_x'$. 
	Since for any~$a\in\widebar{\mathcal{L}}$ we have~$L_a\cong\mathcal{R}(2^{d-1}-3)$ (and~$d\geq 5$), the number of~$2$-sets in~$L_a$ which contain~$x$ is at least~$d-2\geq 3$.
	So in this case the edges in~$\{H\in L_x:H\cup\{x\}\in\mathcal{S}\}$ contribute at least~$\frac{2}{9}\cdot 3=2/3$. In either case, we derive $$\sum_{x\in V_v}w_{\mathcal{F}}(x)\geq 2^d-6+\frac{d-2}{2}+\frac{d-3}{3}\geq 2^d-4 \, .$$
	
	Lastly, if~$\vert 2^{V_v}\setminus\mathcal{F}\vert =6$, then the sets in~$2^{V_v}\setminus\mathcal{F}$ are~$V_v$,~$V_v\setminus\{v\}$,~$A$,~$B$,~$A\setminus\{v\}$, and~$B\setminus\{v\}$ for some distinct~$A,B\in V_v^{(d-1)}$ which both contain~$v$. 
	Thus, for the~$d-3$ vertices~$x\in A\cap B\setminus\{v\}$ there is an~$F_x\in L_x\cap (V\setminus V_v)^{(1)}$ and at least two further~$F_x^i\in L_x$ with~$F_x^i\cap (V\setminus V_v)\neq\emptyset$ and~$\vert F_x^i\vert\leq 2$,~$i\in [2]$. 
	In addition, there are two further vertices~$x\in A\triangle B$ for which there is at least one~$F_x\in L_x\cap (V\setminus V_v)^{(1)}$.  
	For a vertex~$x\in A\cap B\setminus\{v\}$ we observe the following. If~$F_x^i\not\in\mathcal{S}$ for~$i=1,2$, then these two edges together contribute at least~$2/3$ to~$\sum_{x\in V_v}w_{\mathcal{F}}(x)$. 
	If~$F_x^i\in\mathcal{S}$ for some~$i\in\{1,2\}$, then the edges in~$\{H\in L_x:H\cup\{x\}\in\mathcal{S}\}$ contribute at least~$2/3$ as noted above.
	Therefore the definition of the weight entails $$\sum_{x\in V_v}w_{\mathcal{F}}(x)\geq 2^d-7+\frac{d-1}{2}+\frac{2(d-3)}{3}\geq 2^d-4 $$ and thereby the claim is proved if~$c=4$. \end{proof}
    
\section{Further Remarks and Open Problems} \label{sec:conrem}

Consider $m(s)$ to be the following limit introduced in~\cite{FranklWatanabe}
$$m(s) :=\lim_{n\rightarrow \infty} \frac {m(n,s)}{n}\,.$$ 
It is not difficult to check that~$m(s)$ is well-defined (see~\cite{FranklWatanabe}). 
Rephrased by means of this definition, Theorem~\ref{thm:main} implies that for~$c\leq \frac{d}{4}$ we have that
\begin{align}
    m(2^{d-1}-c) = \frac{2^{d}-c}{d}\, . \label{eq:maintheo}
\end{align} 

The first open problem we would like to mention concerns finding a sharp relation between~$d$ and~$c$ such that~(\ref{eq:maintheo}) holds.
More precisely, finding the maximum integer~$c_0(d)$ such that the equality~(\ref{eq:maintheo}) holds for every~$c\leq c_0$. 
In view of Theorem~\ref{thm:main} we have that~$c_0(d)\geq \lfloor \frac{d}{4}\rfloor$, and below we will give a construction that proves that~$c_0(d)\leq d$ for~$d\geq 5$.

Let~$\mathcal F\subseteq 2^V$ with~$\vert V\vert =n$ and~$d$ be a positive integer such that~$d|n$. 
We say that~$\mathcal F$ is~\textit{$d$-local} if there exists a partition of~$V$ into sets of size~$d$ such that every~$F\in \mathcal F$ is a subset of one of the sets of the partition. 
Observe that the extremal construction presented in the proof of Theorem~\ref{thm:main} is a~$d$-local hypergraph with minimum degree~$2^{d-1}-c+1$ and with~$m(n,2^{d-1}-c)+1$ edges. 
That construction can be generalised in the following way.

Take~$d\geq 5$ and~$c\in [2^{d-2}]$ and set~$s=2^{d-1}-c$, for simplicity let~$d|n$.
By definition of~$m(\cdot,\cdot)$, there is a family on~$d$ vertices with~$m(d,s)+1$ edges such that all vertices have degree at least~$s+1$. 
Take~$n/d$ vertex disjoint copies of such a family. 
It is clear that in the resulting family all vertices have degree at least~$s+1$
and the number of edges is~$m(d, s)\frac{n}{d}+1$. 
This family minimises the number of edges for~$d$-local families with minimum degree at least~$s+1$ and gives the following general upper bound on~$m(2^{d-1}-c)$
\begin{align}\label{ineq:constr}
    m(2^{d-1}-c) \leq \frac{m(d, 2^{d-1}-c)}{d}\,. 
\end{align}

Moreover, we observe that for $c=d+1$ we have that $m(d, 2^{d-1}-(d+1)) < 2^d - (d+1)$. 
To see this, consider the family~$\mathcal F \subseteq 2^{[d]}$ containing all sets with at most~$d-2$ vertices. 
Then~$\mathcal F$ has $2^d - (d+1)$ edges and minimum degree $2^{d-1}-d > 2^{d-1}-c$.
Thus, from~(\ref{ineq:constr}) it follows that $$m(2^{d-1}-(d+1)) \leq  \frac{2^{d}-(d+2)}{d}\,.$$ 
This means~(\ref{eq:maintheo}) does not hold for $c=d+1$, and hence~$c_0(d) \leq d$.

Note that this construction is also~$d$-local. An interesting problem is to find the values of~$c$ for which there are no~$d$-local extremal families.

\begin{problem}\label{prob:dlocal}
Given a positive integer~$d\geq 2$, find the minimal~$c_{\star}(d)\in [1,2^{d-2}]$ such that for all~$c \geq c_{\star}$ we have
\begin{align*}
m(2^{d-1}-c) < \frac{m(d, 2^{d-1}-c)}{d}\,.
\end{align*}
\end{problem}

A solution to this problem would give an insight into the structural behaviour of the extremal families: 
For~$c\geq c_{\star}$ and large~$n$ (possibly satisfying certain divisibility conditions) there is no~$d$-local extremal family for~$m(n,2^{d-1}-c)$.
Note that the results in~\cite{Frankl, Watanabe, FranklWatanabe} solved Problem~\ref{prob:dlocal} for $d\leq 4$.


In the following, given a vertex set of size~$n$ we describe a non~$d$-local family that has less edges than any possible~$d$-local family with the same minimum degree. 
More precisely, the construction below yields that, given~$d\geq 5$ and~$c=d$, we have 
\begin{align}\label{ineq:constr2}
    m(2^{d-1} - d) \leq \frac{2^d-d-\frac{1}{2}}{d} < \frac{m(d, 2^{d-1}-d)}{d} \,.
\end{align} 

\begin{constr}\label{constr:nonloc}
\normalfont{
Let $d\geq 5$ and $k$ a positive integer, set~$n=2dk$.
Take $V$ to be a set of $n$ vertices.
Consider~$U_1, \dots ,U_{2k}$ to be a partition of~$V$ into sets of size~$d$, 
and for every set~$U_i$ arbitrarily pick a vertex~$x_i\in U_i$. 
Define \begin{align*}
    \mathcal G 
    &=
    \{S\subseteq V \colon \text{ there is an~$i$ such that }S\subseteq U_i \text{ and } |S|\leq d-2\} \\
    \mathcal H
    &=
    \{U_i\setminus\{x_i\}\colon \text{for  } i\in \{1,2, \dots, 2k \}\}\\ 
    \mathcal I
    &=
    \{ \{x_i, x_{i+1}\} \colon \text{ for } i\in\{1,3, 5, \dots, 2k-1 \}\}\, .
\end{align*}
One can check that the number of edges of the family~$\mathcal F=\mathcal G\cup\mathcal H\cup \mathcal I$ is given by
$$|\mathcal G|+|\mathcal H|+|\mathcal I|= \frac{2^d-d-2}{d}n+1+\frac{n}{d}+\frac{n}{2d} = \frac{2^d-d-\frac{1}{2}}{d}n+1.$$
Moreover, every vertex in $V$ has degree $s=2^{d-1} - d +1$. This implies the first inequality of~\eqref{ineq:constr2}. Taking~$d=c+1$ in Lemma~\ref{lem:locfamiso}~\eqref{it:locnumedges} yields 
$$2^d-d \leq m(d,2^{d-1}-d),$$ 
and thereby the second inequality in~\eqref{ineq:constr2}.
}
\end{constr} 
 
For~$s\leq 16$ (that is~$d\leq 5)$, considering the results from~\cite{Watanabe, Frankl, FranklWatanabe} and Theorem~\ref{thm:secmain} all values of~$m(s)$ are found, except~$m(11)$. 
We recall the conjecture of Frankl and Watanabe~\cite{FranklWatanabe}, which states that Construction~\ref{constr:nonloc} is extremal for~$d=5$. 

\begin{conj}[\cite{FranklWatanabe}]
$m(11)=5.3$
\end{conj}

A complementary approach than the one taken in this paper could be as follows.

\begin{problem}\label{prob:+c} Given a positive integer~$d$ and an integer~$c \in [0, 2^{d-1})$, find the value of~$m(2^{d-1} + c)$.
\end{problem}
 
Naturally, for~$c\geq 2^{d-1}-\frac{d}{4}$ Problem~\ref{prob:+c} is solved by Theorem~\ref{thm:main}. For~$c\leq 2^{d-2}$, the only general result is given in~\cite{FranklWatanabe}, where it is shown that~$m(2^{d-1})= \frac{2^d-1}{d}+\frac{1}{2}$.
For other values of~$c$ Problem~\ref{prob:+c} is still open.

Observe that Theorems~\ref{thm:main} and~\ref{thm:secmain} and the results presented in~\cite{Frankl, Watanabe, FranklWatanabe} concern cases in which $s$ is close to~$2^d$ for some value of $d$. 
In general, there are still large intervals between powers of~$2$ for which the only bounds on~$m(s)$ that are known are those that follow directly from the previously mentioned results.
Finding a solution for Problem~\ref{prob:dlocal} might shed light on this problem by possibly providing a first understanding of the structural behaviour in those intervals.

\end{document}